\documentclass[12pt]{amsart}
\usepackage[margin=0.5in]{geometry}

\newcommand{\R}{\mathbb{R}}
\newcommand{\N}{\mathbb{N}}

\newtheorem{theorem}{Theorem}[section]

\newtheorem{corollary}[theorem]{Corollary}
\newtheorem{lemma}[theorem]{Lemma}
\newtheorem{remark}[theorem]{Remark}

\usepackage{amssymb}

\renewcommand{\le}{\leqslant}

\renewcommand{\ge}{\geqslant}

\begin{document}

\title[Anisotropic nonlocal operators]{Regularity and rigidity theorems \\
for a class of anisotropic nonlocal operators}

\thanks{The authors have been supported by the
ERC grant 277749 ``EPSILON Elliptic
Pde's and Symmetry of Interfaces and Layers for Odd Nonlinearities''}

\author{Alberto Farina}

\address{Alberto Farina:
LAMFA -- CNRS UMR 6140 --
Universit\'e de Picardie Jules Verne --
Facult\'e des Sciences --
33, rue Saint-Leu --
80039 Amiens CEDEX 1, France --
Email:
{\tt alberto.farina@u-picardie.fr}  
}

\author{Enrico Valdinoci}

\address{Enrico Valdinoci:
University of Melbourne --
School of Mathematics and Statistics --
813 Swanston St, Parkville VIC 3010
-- Australia -- and --
University of Western Australia --
School of Mathematics and Statistics --
35 Stirling Highway --
Crawley, Perth
WA 6009, Australia -- and --
Consiglio Nazionale delle Ricerche --
Istituto di Matematica Applicata e Tecnologie Informatiche --
via Ferrata, 1 --
27100 Pavia, Italy -- and --
Universit\`a degli Studi di Milano -- Dipartimento di Matematica --
Via Cesare Saldini, 50 -- 20133 Milano, Italy
-- and --
Weierstra{\ss}-Institut f\"ur Angewandte Analysis und Stochastik --
Mohrenstra{\ss}e, 39 --
10117 Berlin, Germany --
Email:
{\tt enrico@math.utexas.edu} 
}

\begin{abstract}
We consider here operators which are sum of (possibly) fractional
derivatives, with (possibly different) order. 
The main constructive assumption is that the operator is
of order~$2$ in one variable.
By constructing an
explicit barrier, we prove a Lipschitz estimate
which controls the oscillation of the solutions
in such direction with respect to the oscillation of
the nonlinearity in the same direction.

As a consequence, we obtain a rigidity result
that, roughly speaking, states that if the nonlinearity
is independent of a coordinate direction, then so is any global
solution 
(provided that the solution does not grow too much at infinity).
A Liouville type result then follows as a byproduct.
\end{abstract}


\maketitle

\section{Introduction}

Recently a good deal of research has been performed about
nonlocal operators of fractional type, also
in consideration of their probabilistic interpretation
of L\'evy processes. In this framework, it is natural to consider
the superposition of different nonlocal operators
in different directions, possibly with different (fractional)
orders, in relation with the nonlocal diffusive
equations in anisotropic media.\medskip

A first attempt to systematically
study these anisotropic fractional operators
was given in~\cite{SRO, ROV, SROV}.
In particular, the regularity theory of these
anisotropic operators is perhaps harder than expected,
it still presents several open questions 
and some lack of regularity occurs in concrete examples
(for instance, solutions of rather simple equations
with smooth data and smooth domains may fail in this case
to be smooth, see Theorem~1.2 in~\cite{ROV}).
Roughly speaking, the lack of regularity may be caused
by the combination of the nonlocal 
properties of the operator
and the anisotropic structure
of the operator. Namely, first the nonlocal feature may cause
the solution to be only H\"older continuous at the boundary;
then the anisotropic structure may relate the solution in the
interior to values at (or close to) the boundary,
and the nonlocal effect can somehow ``propagate''
the boundary singularity towards the interior, making a smooth
interior regularity theory false in this case
(see~\cite{ROV} for more details about it).
\medskip

The goal of this paper is to provide a very simple approach
to a Lipschitz-type regularity theory for
a family of anisotropic integro-differential
operators, obtained by the superposition
of different operators in different coordinate directions,
and possibly with different order of differentiation.\medskip

The main structural assumption that we take is that
there is one ``special'' coordinate (say the last one)
in which the operator is local and of second order.
In this framework, we will control the derivative
of the solution in this variable by uniform and universal quantities,
depending on the data of the problem.\medskip

More precisely,
the mathematical framework in which we work is the following.
We denote by~$\{e_1,\dots,e_n\}$ the Euclidean base of~$\R^n$.
Given a point~$x\in\R^n$, we use the notation
$$ x=(x_1,\dots,x_n)=x_1 e_1+\dots + x_n e_n,$$
with~$x_i\in\R$. 

We divide the variables of~$\R^n$ into $m$ subgroups of variables,
that is we consider~$m\in\N$ and~$N_1,\dots,N_m\in\N$, with
$N_1+\dots+N_{m-1}=n-1$.
For~$i\in \{1,\dots,m\}$, we use the notation~$N'_i:= N_1+\dots+N_i$,
and we take into account
the set of coordinates
\begin{equation}\label{JK905L:gh}
\begin{split}
& X_1 := (x_1,\dots, x_{N_1}) \in \R^{N_1}\\
& X_2 := (x_{N_1+1},\dots, x_{N_2'}) \in \R^{N_2}\\
& \vdots\\
& X_i := (x_{N_{i-1}'+1},\dots, x_{N_i'}) 
\in \R^{N_i}\\
& \vdots\\
& X_{m-1} := (x_{N_{m-2}'+1},\dots, x_{N_{m-1}'}) 
\in \R^{N_{m-1}}\\
{\mbox{and }}\qquad& X_m:=x_n.
\end{split}\end{equation}
Given~$i\in\{1,\dots,m-1\}$ and~$s_i\in (0,1]$ we will
consider the (possibly fractional) $s_i$-Laplacian 
in the $i$th set of coordinates~$X_i$.
For this scope, given~$y=(y_1,\dots,y_{N_i})\in \R^{N_i}$
it is useful to consider the increment induced by~$y$
with respect to the $i$th set of coordinates in~$\R^n$,
that is one defines
\begin{equation}\label{jU6yhwedcG90:lk}
y^{(i)} := y_1
e_{N_{i-1}'+1}+\dots+
y_{N_i}
e_{N_i'} \in \R^n.\end{equation}
With this notation, one can define
the $N_i$-dimensional 
(possibly fractional) $s_i$-Laplacian
in the $i$th set of coordinates~$X_i$,
for a (smooth) function~$u:\R^n\to\R$ by
\begin{equation}\label{FL} (-\Delta_{X_i})^{s_i} u(x) :=
\left\{
\begin{matrix}
-\partial^2_{x_{N_{i-1}'+1}} u(x)-\,\dots\,-
\partial^2_{x_{N_i'}} u(x)& {\mbox{ if }}s_i=1, \\
&\\ c_{N_i, s_i}
\displaystyle\int_{\R^{N_i}} \frac{2u(x)-u(x+y^{(i)})-u(x-y^{(i)})
}{|y^{(i)}|^{N_i+2s_i}}\,dy^{(i)}
& {\mbox{ if }}s_i\in(0,1),
\end{matrix}
\right. \end{equation}
The quantity~$c_{N_i,s_i}$ in~\eqref{FL}
is just a positive normalization constant, whose explicit\footnote{We
wrote the value of~$c_{N,s}$ as in~\eqref{cns} to be consistent
with the literature, see e.g.
notation of~\cite{dyda}. Of course, such value
can be equivalently written in other forms, according
to the different tastes.
The explicit value of the normalization constant
in~\eqref{cns} 
plays no major role in this paper, but
it is useful for consistency properties as~$s_i\to 1$.}
value for~$N\in\N$ and $s\in(0,1)$ is taken to be
\begin{equation}\label{cns}
c_{N,s}:=
\frac{2^{2s-1}\,\Gamma(s+\frac{N}{2})}{\pi^{\frac{N}{2}}\,|\Gamma(-s)|},\end{equation}
where~$\Gamma$ is the Euler's Gamma Function.
We refer to~\cite{land, silv, guida} and to the references
therein for further motivations about fractional operators.
\medskip

In this paper we consider a
pseudo-differential operator, which is the sum of
(possibly) fractional Laplacians in the different coordinate directions~$X_i$,
with~$i\in\{1,\dots,m-1\}$, plus
a local second derivative in the direction~$x_n$.
The operators involved may have
different orders and they may be multiplied
by possibly different coefficients:
that is, given~$a_1,\dots,a_{m-1}\ge0$ and~$a>0$,
we define
\begin{equation}\label{L} 
\begin{split}
L\,&:=\sum_{i=1}^{m-1} a_i (-\Delta_{X_i})^{s_i} -a\partial^2_{x_n}\\
&=\sum_{i=1}^{m} a_i (-\Delta_{X_i})^{s_i},\end{split}\end{equation}
where in the latter identity
we used the convention that~$a_m:=a$, $s_1,\dots,s_{m-1}\in(0,1]$
and~$s_m:=1$.

Given the operator in~\eqref{L}, we stress that a very 
important structural difference
with respect to the classical local case is that
fractional objects are in general not reduced to
the sum of their directional components\footnote{That is,
if~$x=(x_1,\dots,x_n)\in\R^n$ the following formulas
are false, unless~$s=1$:
\begin{eqnarray*}
&&(-\partial^2_{x_1})^s +\dots+(-\partial^2_{x_n})^s
=(-\Delta_{x})^s\\
{\mbox{and }}&&
(-\Delta_{(x_1,\dots,x_N)})^s+
(-\Delta_{(x_{N+1},\dots,x_{N+K})})^s
=(-\Delta_{(x_1,\dots,x_{N+K})})^s
.\end{eqnarray*}}
\medskip

The main result that we prove in this paper
is a Lipschitz regularity theory in the last coordinate variable
that extends the one of~\cite{brandt69} (which was obtained
in the classical setting of local operators).
To this goal, we denote by~$B^N_r$ the open ball
of~$\R^N$ centered at the origin and with radius~$R$.
Also,
given~$d_1,\dots,d_m>0$, we set~$d:=(d_1,\dots,d_m)$
and
$$ Q_d := B^{N_1}_{d_1} \times\dots\times
B_{d_{m-1}}^{N_{m-1}}
\times (-d_m,d_m) =\prod_{i=1}^m B_{d_i}^{N_i},$$
where in the latter identity
we used the convention that~$N_m:=1$.

Also, given~$\kappa>0$
we denote by~$Q_{d,\kappa}$ the dilation of factor~$\kappa$
in the last coordinate (leaving the others fixed), 
that is
$$ Q_{d,\kappa}:=
B^{N_1}_{d_1} \times\dots\times
B_{d_{m-1}}^{N_{m-1}}
\times (-\kappa d_m,\kappa d_m).$$
Of course, by construction~$Q_{d,1}=Q_d$.
In accordance with
the constant fixed in~\eqref{FL},
it is also convenient to introduce the following notation\footnote{We 
\label{etaii}
observe that~$\eta_i=1/(2N_i)$ if~$s_i=1$, since~$\Gamma\left(
1+\frac{N_i}{2}\right) = \frac{N_i}{2}\,\Gamma\left(\frac{N_i}{2}\right)$.}
for a suitable universal quantity, for any~$i\in \{1,\dots,m\}$:
\begin{equation}\label{etai}
\eta_i:=
\frac{ \Gamma\left(\frac{N_i}{2}\right)}{2^{2s_i}\,
\Gamma(s_i+1)\,
\Gamma\left(s_i+\frac{N_i}{2}\right)}.\end{equation}
With this notation, we have the following result:

\begin{theorem}\label{MAIN}
Let~$f:Q_{d,2}\to\R$ and~$u:\R^n\to\R$ be a solution of~$Lu=f$ in~$Q_{d,2}$.
Then, for any~$t\in (-d_m,d_m)$,
\begin{equation}\label{STIMA:MAIN}
\frac{|u(te_n)-u(-te_n)|}{|t|} \le
\frac{ d_m\,{\mathfrak{S}} }{a}
+\frac{ \tilde C\,d_m \,\|u\|_{L^\infty(\R^{n})}}{
\displaystyle\min_{i\in\{1,\dots,m\}} (\eta_i d_i^{2s_i}) },
\end{equation}
where
\begin{eqnarray*} 
{\mathfrak{S}} &:=&
\sup_{ (x,t)\in Q_d\times(0,d_m)} |f(x+te_n)-f(x-te_n)|,
\\{\mbox{and }}\quad
\tilde C&:=& 
\frac{2(a_1+\dots+a_m)}{a} + 1.\end{eqnarray*}
\end{theorem}

Higher regularity results (for different types of nonlocal
anisotropic operators)
have been obtained in~\cite{SRO, ROV}
(indeed, general anisotropic operators
can be considered in~\cite{SRO, ROV},
but only the kernel with the same homogeneity were taken into account).
Some advantages are offered by Theorem~\ref{MAIN}
with respect to the other results available in the literature.
First of all, Theorem~\ref{MAIN} comprises the case of operators
of different orders (e.g. the $s_i$ can be all different
and both local and nonlocal operators can be superposed).
Moreover, Theorem~\ref{MAIN} may select the ``local'' coordinate direction
independently on the others, in order to take into account
the behavior of the nonlinearity in this single coordinate and detect
its effect on the oscillation of the solution (notice in particular
the term~${\mathfrak{S}}$
appearing in
Theorem~\ref{MAIN}, which only depends on the oscillation of~$f$
in the last coordinate direction). 
As a matter of fact, the diffusive operators in the other variables
can also degenerate (indeed~$a_i$ may vanish for some~$i\in\{1,\dots,m-1\}$).

In addition,
all the constants appearing in Theorem~\ref{MAIN} can be computed
explicitly without effort and the proof is rather simple and
it makes use only of one explicit barrier
(the barrier will be given in formula~\eqref{BA}
and, as a matter of fact, this argument may be seen as the fractional
counterpart of the regularity theory developed by~\cite{brandt69}
in the local framework).\medskip

As a technical remark, we point out that, for simplicity,
the notion of solution in Theorem~\ref{MAIN} is taken in the classical
sense, i.e. the function~$u$ will be implicitly assumed to be smooth
enough to compute the operator~$L$ pointwise (in this sense,
formula~\eqref{STIMA:MAIN}
reads as an ``a priori estimate''). Nevertheless, the same
argument that we present goes through, for instance, by applying the
operator to smooth functions that touch the
solution from above/below, that is one can assume simply that
the solution in Theorem~\ref{MAIN} is taken in the viscosity sense
(in this case,
formula~\eqref{STIMA:MAIN}
reads as an ``improvement of regularity'').\medskip

We point out that, as a simple consequence of
Theorem~\ref{MAIN}, we obtain an interior Lipschitz estimate
in the last variable:

\begin{corollary}\label{BR}
Let~$u:\R^n\to\R$ be a solution of~$Lu=f$ in~$B_1$.
Then
\begin{equation}\label{8ihn6tyufryury7}
\|\partial_{x_n} u\|_{L^\infty(B_{1/2})}\le C\,\big(
\|f\|_{L^\infty(B_1)}+\|u\|_{L^\infty(\R^n)}\big),\end{equation}
for some~$C>0$, depending on~$a_1,\dots,a_m$, $s_1,\dots,s_{m-1}$,
and~$N_1,\dots,N_{m-1}$.
\end{corollary}

As a matter of fact, when~$s_1=\dots=s_m=1$,
Corollary~\ref{BR} reduces to the classical Lipschitz
regularity theory as presented in~\cite{brandt69}.\medskip

We observe that the regularity results obtained
in this paper can be also combined efficiently
with other results available in the literature,
possibly leading to higher regularity results.
To make a simple example of this feature, we give the following result:

\begin{corollary}\label{RRO}
Let~$s\in(0,1)$, $a_1,\dots,a_{m-1},\,a>0$ and
\begin{equation}\label{OP:56A} 
L_* :=\sum_{i=1}^{m-1} a_i (-\Delta_{X_i})^{s} -a\partial^2_{x_n}.\end{equation}
Let~$f\in L^\infty(\R^n)$ be Lipschitz continuous in~$B_1$ with respect to
the variable~$x_n$.
Let~$u:\R^n\to\R$ be a solution of~$L_* u=f$ in~$B_1^{n-1}\times\R$.
Then
\begin{equation*}
\|u\|_{C^\gamma (B_{1/2})}\le C\,\big(
\|f\|_{L^\infty(\R^n)}+
\|\partial_{x_n}f\|_{L^\infty(B_1)}+
\|u\|_{L^\infty(\R^n)}\big),\end{equation*}
where
$$ \gamma:=\left\{
\begin{matrix}
2s & {\mbox{ if }} s< 1/2,\\
1-\epsilon & {\mbox{ if }} s\ge 1/2
\end{matrix}
\right.$$
for some~$C>0$, depending on~$a_1,\dots,a_m$, $s$,
and~$N_1,\dots,N_{m-1}$ (with the caveat that
when~$s\ge 1/2$, one can choose~$\epsilon$
arbitrarily in~$(0,1)$ and~$C$ will also depend on~$\epsilon$).
\end{corollary}

Another interesting consequence of Theorem~\ref{MAIN}, is also
the following rigidity result, valid when all the fractional
exponents are larger than~$1/2$:

\begin{theorem}\label{MAIN 2}
Let~$f:\R^n\to\R$.
Assume that
$$ \sigma:= 2\min\{s_1,\dots,s_n\}-1>0.$$
Let~$u:\R^n\to\R$ be a solution of~$Lu=f$ in the whole of~$\R^n$.
Assume that~$f$ does not depend on the $n$th coordinate
and that\footnote{As customary, the notation in~\eqref{XRT7}
simply means that
$$ \lim_{R\to+\infty} \frac{ \|u\|_{L^\infty(B_R)} }{R^{\sigma}}=0.$$}
\begin{equation}\label{XRT7}
\|u\|_{L^\infty(B_R)} = o(R^{\sigma}) \end{equation}
as~$R\to+\infty$.

Then~$u$ does not depend on the $n$th coordinate.
\end{theorem}

\begin{remark}\label{RK1}{\rm A simple, but interesting, consequence
of Theorem~\ref{MAIN 2} is that if~\eqref{XRT7} holds and~$f$
is identically zero, then
$$ L_*u=0\qquad{\mbox{ in~$\R^{n-1}$,}} $$
where we used the notation in~\eqref{OP:56A}.
Therefore, if~$L_*$ enjoys a Liouville property,
then~$u$ is necessarily constant.

This feature holds, in particular, when~$L_*=
(-\partial^2_{x_1})^s +\dots+(-\partial^2_{x_{n-1}})^s$,
see Theorem~2.1 in~\cite{SRO}.}\end{remark}

\begin{remark}{\rm The observation in Remark~\ref{RK1}
also says that, if~\eqref{XRT7} is
satisfied and~$L_*$ enjoys a Liouville property,
then the problem~$Lu=f$ possesses a unique solution,
up to an additive constant.
}\end{remark}

The rest of the paper is organized as follows. 
The proof of Theorem~\ref{MAIN}, based on the barrier method of~\cite{brandt69},
is contained in Section~\ref{S:1}. Then, in Section~\ref{S:2}, we combine
our results with those of~\cite{SRO} and we prove Corollary~\ref{RRO}.
The proof of Theorem \ref{MAIN 2}, which combines our result with
a cutoff argument, is contained in Section~\ref{S:3}.

\section{Proof of Theorem~\ref{MAIN}}\label{S:1}

\subsection{A useful explicit barrier}

We recall here a useful barrier. Here and in what follows we
use the standard ``positive part'' notation for any~$t\in\R$, i.e.
$$ t_+ :=\max\{t,0\}.$$
We will also exploit the notation in~\eqref{JK905L:gh}
and~\eqref{etai}.

\begin{lemma}\label{DY}
Let~$s_i\in(0,1]$ and~$d_i>0$.
For any~$x=(X_1,\dots,X_{m-1},x_n)\in\R^n$ let
$$ \Phi_{d_i}(z):=\eta_i\,
(d_i^2-|X_i|^2)_+^{s_i}.$$
Then,
for any~$x\in\R^n$ with~$X_i\in B_{d_i}^{N_i}$,
we have that
$$ (-\Delta_{X_i})^{s_i} \Phi_{d_i}(x)=1.$$
\end{lemma}

\begin{proof}
The result is obvious for~$s_i=1$
(recall the footnote on page~\pageref{etaii}), hence we suppose~$s_i\in(0,1)$.
We let~$\Psi_{d_i}:=\eta_i^{-1}\,
\Phi_{d_i}(z)=
(d_i^2-|X_i|^2)_+^{s_i}$.
By scaling variables~$x_*:=\frac{x}{d_i}$,
$X_{*,i}:=\frac{X_i}{d_i}$
and~$\zeta_*:=\frac{\zeta}{d_i}$,\
we obtain that
\begin{eqnarray*}
(-\Delta)^{s_i} \Psi_{d_i}(x)&=& c_{N_i,s_i} \int_{\R^{N_i}}
\frac{2(d_i^2-|X_i|^2)_+^{s_i} - (d_i^2-|X_i+\zeta|^2)_+^{s_i} -(d_i^2-|X_i-\zeta|^2)_+^{s_i}}{
|\zeta|^{N_i+2s_i}}\,d\zeta \\
&=&
c_{N_i,s_i} \,d_i^{2s_i} \,d_i^{N_i}\,\int_{\R^{N_i}}
\frac{2(1-|X_{*,i}|^2)_+^{s_i} - (1-
|X_{*,i}+\zeta_*|^2)_+^{s_i} -(1-
|X_{*,i}-\zeta_*|^2)_+^{s_i} }{
d_i^{N_i+2s_i} |X_{*,i}|^{N_i+2s_i}} \,d\zeta_*\\
&=& (-\Delta_{X_i})^{s_i} \Psi_{1}(x_*)\\
&=& \frac{2^{2s_i}\,
\Gamma(s_i+1)\,
\Gamma\left(s_i+\frac{N_i}{2}\right)
}{\Gamma\left(\frac{N_i}{2}\right)},
\end{eqnarray*}
see for instance Table~3 of \cite{dyda} for the last identity
(here, we used the notation~$ \Psi_{1}$ to denote~$ \Psi_{d_i}$ when~$d_i=1$).
\end{proof}

\subsection{Completion of the proof of Theorem~\ref{MAIN}}\label{COMP}

For any~$t\in\R$, we define\footnote{As a technical
remark, we point out that the assumption
that the operator is
``local'' in the last coordinate is used at this
point, 
since if~$t>0$ we have that~$u^\pm(x,t)
=u(x\pm te_n)$,
and so, if we differentiate with respect to $t$
in the domain~$\{t\in(0,d_m)\}$, we have that~$\partial_t u^\pm(x,t)
=\pm \partial_{x_n} u(x\pm te_n)$.}
\begin{equation}\label{upm}
u^\pm(x,t):= u(x\pm t_+ e_n)=u(x_1,\dots,x_{{n-1}}, x_{n}\pm t_+).\end{equation}
Similarly, we define~$f^\pm(x,t):= f(x\pm t_+ e_n)$.
Let also
$$ v(x,t):=u^+(x,t)-u^-(x,t)\;
{\mbox{ and }}\;g(x,t):=f^+(x,t)-f^-(x,t).$$
We fix~$\nu\in(0,a)$ (to be taken as close to~$a$
as we wish in what fallows).
Recalling \eqref{L}, we introduce the operator
\begin{equation}\label{LSTA}\begin{split}
L_* \,:=\;& L +\nu \partial^2_{x_n} - \nu \partial^2_t \\
=\;&\sum_{i=1}^{m} a_i (-\Delta_{X_i})^{s_i} - \nu\,(-\partial^2_{x_n})
- \nu \partial^2_t\\
=\;& 
\sum_{i=1}^{m-1} a_i (-\Delta_{X_i})^{s_i} -(a-\nu) \partial^2_{x_n}
- \nu \partial^2_t.\end{split}\end{equation}
Notice that~$L_*$ is an operator with one 
variable more than~$L$
(namely, the new variable~$t\in\R$). We claim that
\begin{equation}\label{P6}
L_* v=g
\end{equation}
for any~$(x,t)\in Q_d\times(0,d_m)$.
To check~\eqref{P6}, we first
notice that if~$(x,t)\in Q_d\times(0,d_m)$
then~$x\pm te_n\in Q_{d,n,2}$, and we know that~$Lu=f$
in the latter set. Also, since the (fractional) Laplacian is
translation invariant,
\begin{equation}\label{fora0}
(-\Delta_{X_i})^{s_i} u^\pm = \Big(
(-\Delta_{X_i})^{s_i} u\Big)^\pm\end{equation}
for any $i\in\{1,\dots,m\}$ and
any~$(x,t)\in Q_d\times(0,d_m)$
(notice that the variable~$t$
plays the role of a fixed parameter here).
Moreover
\begin{equation}\label{fora}
\partial^2_{t} u^\pm =\Big( \partial^2_{x_n} u\Big)^\pm
\end{equation}
for any~$(x,t)\in
Q_d\times(0,d_m)$.
In turn, we see that~\eqref{LSTA}, \eqref{fora0} and~\eqref{fora} imply that
$$ L_* u^\pm = (Lu)^\pm$$
and thus, by linearity,
$$L_* v=L_*(u^+-u^-) =(Lu)^+-(Lu)^-=f^+-f^-=g,$$
which
establishes~\eqref{P6}.
Now we set
\begin{equation}\label{0ighc0}
\begin{split}
& c_o := \sum_{i=1}^m \eta_i d_i^{2s_i} + \frac{d_m^{2}}{2} =
\sum_{i=1}^{m-1} d_i^{2s_i} + d_m^{2}
\\{\mbox{ and }} &\qquad
A_0:=\sum_{i=1}^m a_i.\end{split}\end{equation}
Let also
\begin{equation}\label{0ighc0-bis}
\begin{split}
& A_1:=A_0A_2 +
\|g\|_{L^\infty(Q_d\times(0,d_m))} +(a-\nu),\\
{\mbox{where }}\qquad
&A_2 := \frac{ \|v\|_{L^\infty(\R^{n+1})}}{
\displaystyle\min_{i\in\{1,\dots,m\}} (\eta_i d_i^{2s_i}) }.\end{split}\end{equation}
We consider the barrier
\begin{equation}\label{BA}
\begin{split}
& \Phi(x,t):= \frac{A_1}{\nu} \Phi_1(t)+A_2\Phi_2(x,t)\\
{\mbox{with }}\;& \Phi_1 (t):=\frac{t_+ \,(d_m -t)_+ }{2}\\
{\mbox{and }}\;
&\Phi_2(x,t)=\Phi_2(X_1,\dots,X_{m-1},X_m,t):= c_o
-\sum_{i=1}^m\eta_i (d_i^2-|X_i|^2)^{s_i}_+ 
- \frac{(d_m^2-t^2)_+}{2}.\end{split}
\end{equation}
Notice that~$\Phi_1\ge0$ and also, by~\eqref{0ighc0},
$$\Phi_2\ge 
c_o-\sum_{i=1}^m\eta_i d_i^{2s_i}- \frac{d_m}{2}=0.$$
Consequently,
\begin{equation}\label{ew1}
\Phi\ge0
\end{equation}
Moreover, using the notation
of Lemma~\ref{DY}, we see that
$$ \Phi_2(x,t)=
c_o-\sum_{i=1}^n
\Psi_{d_i}(x_i)
- \Psi_{d_m}(t).$$
Therefore, making use of
Lemma~\ref{DY},
we conclude that, for any~$(x,t)\in Q_d\times(0,d_m)$,
\begin{eqnarray*}
L_* \Phi(x,t)&
=&
\sum_{i=1}^{m} a_i (-\Delta_{X_i})^{s_i} 
\Phi(x,t)
+\nu \partial^2_{x_n}\Phi(x,t)
- \nu \partial^2_t\Phi(x,t)\\
&=& A_2
\sum_{i=1}^{m} a_i A_2 (-\Delta_{X_i})^{s_i} 
\Phi_2 (x,t)
+A_2\nu\partial^2_{x_n}\Phi_2(x,t)
- \nu \partial^2_t\Phi(x,t)\\
&=&
-A_2 \sum_{i=1}^{m} a_i -A_2\nu
+\nu \left( \frac{A_1}{\nu} +A_2\right)
\\ &=&
-A_2 \sum_{i=1}^{m} a_i +A_1.\end{eqnarray*}
That is, recalling~\eqref{0ighc0}
and~\eqref{0ighc0-bis},
\begin{equation}\label{SI}
\begin{split} &L_* \Phi(x,t) = -A_2 A_0 +A_1
=
\|g\|_{L^\infty(Q_d\times(0,d_m))} +(a-\nu)
\\ &\qquad\ge \pm g(x,t)+(a-\nu)=\pm L_* v(x,t)+(a-\nu),\end{split}\end{equation}
for any~$(x,t)\in Q_d\times(0,d_m)$,
where we used~\eqref{P6} in the last step.

Now we claim that
\begin{equation}\label{P7}
{\mbox{$\Phi(x,t)\pm v(x,t)\ge0$ for any~$(x,t)\in \R^{n+1} \setminus\big(
Q_d\times(0,d_m)\big)$.}}
\end{equation}
To check this, we take~$(x,t)$
outside~$Q_d\times(0,d_m)$,
and we distinguish three cases:
either~$t\le0$, or~$t\ge d_m$,
or~$x\in\R^n\setminus Q_d$.

First, when~$t\le0$, we have that~$t_+=0$, so
we use~\eqref{upm} to
see that~$u^+(x,0)=u(x)=u^-(x,0)$ and so~$v(x,0)=0$.
Then~$\pm v(x,0)=0\le\Phi(x,0)$ in this case, thanks to~\eqref{ew1}
and this establishes~\eqref{P7} when~$t\le0$.

Now, let us deal with the case in which~$t\ge d_m$.
In this case~$(d_m^2-t^2)_+=0$, hence, by~\eqref{BA},
\begin{eqnarray*}
\Phi(x,t)&\ge& A_2\Phi_2(x,t)\\ &=&
A_2\left[ c_o
-\sum_{i=1}^m \eta_i \,(d_i^2-|X_i|^2)^{s_i}_+\right]\\&\ge&
A_2\left[ c_o
-\sum_{i=1}^m\eta_i d_i^{2s_i}\right] \\
&=& \frac{ A_2 d_m^{2} }{2}\\
&\ge& \|v\|_{L^\infty(\R^{n+1})}\\
&\ge& \pm v(x,t),\end{eqnarray*}
as desired. It remains to consider the case~$x\in
\R^n\setminus Q_d$. Under this circumstance, we have that
there exists~$i_o\in \{1,\dots,m\}$ such that~$|X_{i_o}|\ge d_{i_o}$.
Accordingly
$$ \sum_{i=1}^m (d_i^2-|X_i|^2)^{s_i}_+ =
\sum_{ {1\le i\le m}\atop{i\ne i_o}} (d_i^2-|X_i|^2)^{s_i}_+
\le \sum_{ {1\le i\le m}\atop{i\ne i_o}} d_i^{2s_i},$$
and so
\begin{eqnarray*}
\Phi(x,t)&\ge& A_2\Phi_2(x,t)\\ &\ge&
A_2\left[ c_o -
\sum_{ {1\le i\le m}\atop{i\ne i_o}}\eta_i d_i^{2s_i}
- \frac{d_m^{2}}{2}\right] \\
&=& A_2 \eta_{i_o} d_{i_o}^{2s_{i_o}}
\\
&\ge& \|v\|_{L^\infty(\R^{n+1})}\\
&\ge& \pm v(x,t),
\end{eqnarray*}
which completes the proof of~\eqref{P7}.

Now we show that the inequality in~\eqref{P7} propagates
inside~$Q_d\times(0,d_m)$, namely that
\begin{equation}\label{P8}
{\mbox{$\Phi(x,t)\pm v(x,t)\ge0$ for any~$(x,t)\in\R^{n+1}$.}}
\end{equation}
The proof of~\eqref{P8} is mostly Maximum Principle. The details
are as follows. Suppose, by contradiction,
that~\eqref{P8} were false. 
Then we set~$h:=\Phi\pm v$. Notice that,
since~$u$ is assumed to be continuous,
so is~$h$, due to~\eqref{upm}, and then~\eqref{P7} would imply that
$$ \min_{\overline{Q_d\times(0,d_m)}} h=:\mu< 0.$$
Let~$\bar p:=(\bar x,\bar t)$ attaining the minimum of~$h$, that is
$$ (-\infty,0)\ni \mu=h(\bar p)\le h(\xi),$$
for any~$\xi\in\R^{n+1}$. By~\eqref{P7},
we have that~$\bar p$ lies in~$Q_d\times(0,d_m)$,
hence, by~\eqref{SI},
\begin{equation}\label{MP}
L_* h(\bar p)\ge a-\nu>0. \end{equation}
On the other hand, 
recalling the notation in~\eqref{jU6yhwedcG90:lk},
for any~$i\in\{1,\dots,m-1\}$ and any~$y\in\R^{N_i}$,
we have that
$$ 2h(\bar p) -h(\bar p+y^{(i)})-h(\bar p-y^{(i)}) \le0,$$
due to the minimality of~$\bar p$. Similarly~$(-\partial_{x_j}^2) h(\bar p)\le0$
for any~$j\in\{1,\dots,n\}$,
as well as~$(-\partial^2_t)h(\bar p)\le0$.
Therefore~$(-\Delta_{X_i})^{s_i}(\bar p)\le0$,
for any~$i\in\{1,\dots,m\}$.
Consequently, by~\eqref{LSTA}, we infer that~$L_* h(\bar p)\le0$.
The latter inequality is in contradiction with~\eqref{MP}
and thus we have proved~\eqref{P8}.

By choosing the sign in~\eqref{P8}, we deduce that
\begin{equation}\label{P8bis}
{\mbox{$|v(x,t)|\le\Phi(x,t)$
for any~$(x,t)\in\R^{n+1}$.}}\end{equation}
Moreover, recalling~\eqref{BA}
and~\eqref{0ighc0}, for any~$t\in(0,d_m)$,
\begin{eqnarray*}
\Phi_2(0,t)&=&
c_o -\sum_{i=1}^m\eta_i d_i^{2s_i}
- \frac{(d_m^2-t^2)_+}{2} \\
&=& \frac{d_m^{2}}{2}- \frac{(d_m^2-t^2)}{2}\\
&=& \frac{t^2}{2}.\end{eqnarray*}
In addition, 
$$ \Phi_1 (t)\le d_m t_+,$$
therefore, by~\eqref{BA}, for any~$t\in(0,d_m)$,
$$\Phi(0,t)\le 
\frac{A_1 d_m t}{\nu} +\frac{A_2 t^2}{2}.$$
This and~\eqref{P8bis} imply that, for any~$t\in(0,{d_m})$,
\begin{eqnarray*}
&& \frac{|u(te_n)-u(-te_n)|}{|t|}=
\frac{|u^+(0,t)-u^-(0,t)|}{t}=
\frac{|v(0,t)|}{t}\le\frac{\Phi(0,t)}{t}
\\ &&\qquad\le
\frac{A_1\,d_m}{\nu} +
\frac{A_2\, t}{2}
\\&&\qquad= 
\frac{ \big[ A_0A_2 +
\|g\|_{L^\infty(Q_d\times(0,d_m))} +(a-\nu) \big]\,d_m}{\nu}
+\frac{A_2\, t}{2}.\end{eqnarray*}
Now we observe that the first term in the above inequality
remains unchanged if we replace~$t$ with~$-t$, and therefore
the inequality is valid for any~$t\in (-d_m,d_m)$.
Furthermore, we can now take~$\nu$ as close to~$a$ as we wish
(recall that~$A_0$ and~$A_2$ do not depend on~$\nu$),
hence we obtain that, for any~$t\in (-d_m,d_m)$,
\begin{eqnarray*}
\frac{|u(te_n)-u(-te_n)|}{|t|^{s_n}}
&\le&
\frac{\big[ A_0A_2 +
\|g\|_{L^\infty(Q_d\times(0,d_m))} \big]\,d_m}{a}
+\frac{A_2\, t}{2} \\
&\le& \frac{ \big[ A_0A_2 +
\|g\|_{L^\infty(Q_d\times(0,d_m))} \big]\,d_m}{a}
+\frac{A_2\, d_m}{2}\\
&=&
\frac{\|g\|_{L^\infty(Q_d\times(0,d_m))} \,d_m}{a}
+ A_2\,d_m\left(
\frac{ A_0}{a}+\frac{1}{2}\right)\\
&=& \frac{\|g\|_{L^\infty(Q_d\times(0,d_m))}\,d_m }{a}
+ \frac{ \|v\|_{L^\infty(\R^{n+1})} \,d_m}{
\displaystyle\min_{i\in\{1,\dots,m\}} (\eta_i d_i^{2s_i}) }
\left(\frac{ a_1+\dots+a_m}{a}+\frac{1}{2}\right).
\end{eqnarray*}
This completes the proof of Theorem~\ref{MAIN}.

\section{Proof of Corollary~\ref{RRO}}\label{S:2}

The proof combines Corollary~\ref{BR} here with Theorem~1.1(a)
in~\cite{SRO}.
To this goal, fixed~$t\in \left[-\frac1{1000},\frac1{1000}\right]$
(to be taken arbitrarily small in the sequel) we define
\begin{equation}\label{67utgr67rfoyhHH}
u_\sharp(x):= \frac{u(x+te_n)-u(x)}{t}\;{\mbox{ and }}\;
f_\sharp(x):= \frac{f(x+te_n)-f(x)}{t}.\end{equation}
By formula~\eqref{8ihn6tyufryury7} in Corollary~\ref{BR}, we already know that
\begin{equation}\label{97iuhtYUIO99}
\|u_\sharp\|_{L^\infty(\R^n)}\le
\|\partial_{x_n} u\|_{L^\infty(\R^n)}\le C\,\big(
\|f\|_{L^\infty(\R^n)}+\|u\|_{L^\infty(\R^n)}\big).\end{equation}
Also, we point out that~$L_* u_\sharp = f_\sharp$ in~$B_{99/100}$,
and so, using again Corollary~\ref{BR},
$$ \|\partial_{x_n} u_\sharp \|_{L^\infty(B_{97/100})}\le C\,\big(
\|f_\sharp\|_{L^\infty(B_{99/100})}+\|u_\sharp \|_{L^\infty(\R^n)}\big).$$
This, combined with~\eqref{97iuhtYUIO99}, gives that
\begin{eqnarray*} &&\sup_{x\in B_{97/100}}
\left| \frac{\partial_{x_n} u(x+te_n)-\partial_{x_n}u(x)}{t} \right|=
\|\partial_{x_n} u_\sharp \|_{L^\infty(B_{97/100})}\\ &&\qquad\le C\,\big(
\|\partial_{x_n} f\|_{L^\infty(B_1)}+
\|f\|_{L^\infty(\R^n)}+\|u\|_{L^\infty(\R^n)}\big).\end{eqnarray*}
Hence, taking~$t$ to the limit,
\begin{equation}\label{97iuhtYUIO99-BIS}
\sup_{x\in B_{97/100}} |\partial^2_{x_n} u(x)|\le
C\,\big(
\|\partial_{x_n} f\|_{L^\infty(B_1)}+
\|f\|_{L^\infty(\R^n)}+\|u\|_{L^\infty(\R^n)}\big).
\end{equation}
Now, given~$i\in \{1,\dots,m-1\}$ we consider
the sphere~$S^{N_i-1}$ in the Euclidean space~$\R^{N_i}$
(of course, if~$N_i=1$, then~$S^{N_i-1}$ reduces to
two points).

We also set~$S^{n-2}:=\{(x_1,\dots,x_{n-1}) {\mbox{ s.t. }}
x_1^2+ \dots+x_{n-1}^2=1\}$ and we observe that
each~$S^{N_i-1}$ is naturally immersed into~$S^{n-2}$
(in the same way as~$\R^{N_i-1}$ is immersed into~$\R^{n-1}$).

We denote by~${\mathcal{H}}_i$ the $(N_i-1)$-dimensional
Hausdorff measure restricted to~$S^{N_i-1}$
(if~$N_i=1$, we replace it by the Dirac's delta on the two points
given by~$S^{N_i-1}$).
Then we consider the measure
$$ \mu := \sum_{i=1}^{m-1} \frac{a_i\,c_{N_i,s_i}}{2} {\mathcal{H}}_i.$$
We fix~$\tilde x_n\in \left[ -\frac1{100},\frac1{100}\right]$ and
we set
\begin{eqnarray*}
&& \tilde u(x_1,\dots,x_{n-1}):= u(x_1,\dots,x_{n-1},\tilde x_n)\\
{\mbox{and }}&&
\tilde f(x_1,\dots,x_{n-1}):= f(x_1,\dots,x_{n-1},\tilde x_n)
+a\partial^2_{x_n} u(x_1,\dots,x_{n-1},\tilde x_n)
.\end{eqnarray*}
We use \eqref{FL} and polar coordinates
on~$\R^{N_i}$ to see
that, for any~$\tilde x=(x_1,\dots,x_{n-1})\in B^{n-1}_{99/100}$,
\begin{eqnarray*}
\tilde L \tilde u(\tilde x) &:=&
\int_{S^{n-2}} \left[ \int_\R
\Big( \tilde u(\tilde x+\theta r)+ \tilde u(\tilde x-\theta r)
-2\tilde u(\tilde x)\Big) \frac{dr}{|r|^{1+2s}}
\right]\,d\mu(\theta) \\
&=&
\sum_{i=1}^{m-1} \frac{ a_i\,c_{N_i,s_i} }{2}
\int_{S^{N_i-1}} \left[ \int_\R\Big( \tilde u(\tilde x+\theta r)+
\tilde u(\tilde x-\theta r)
-2\tilde u(\tilde x)\Big) \frac{dr}{r^{1+2s}}
\right]\,d{\mathcal{H}}_i(\theta)\\
&=&
\sum_{i=1}^{m-1} a_i\,c_{N_i,s_i}
\int_{S^{N_i-1}} \left[ \int_0^{+\infty}\Big( \tilde u(\tilde x+\theta r)+ \tilde u(\tilde x-\theta r)
-2\tilde u(\tilde x)\Big) \frac{dr}{r^{1+2s}}
\right]\,d{\mathcal{H}}_i(\theta)
\\ &=&
\sum_{i=1}^{m-1} a_i\,c_{N_i,s_i}
\int_{\R^{N_i}} \frac{\tilde u(\tilde x+y^{(i)})+\tilde u(\tilde x-y^{(i)})
-2\tilde u(\tilde x) }{|y^{(i)}|^{N_i+2s}} \,dy^{(i)}
\\ &=& \sum_{i=1}^{m-1} a_i\,(-\Delta_{X_i})^{s_i}\tilde u(\tilde x)
\\ &=&
L_* \tilde u(\tilde x) +a\partial^2_{x_n} u(\tilde x,\tilde x_n)
\\ &=& f(\tilde x, \tilde x_n) +a\partial^2_{x_n} u(\tilde x,\tilde x_n)
\\ &=& \tilde f(\tilde x).
\end{eqnarray*}
Notice that, with this setting, the operator~$\tilde L $ satisfies
formula~(1.1) in~\cite{SRO}.

Furthermore, we have that
\begin{equation}\label{HasUIYa66}
\inf_{\nu\in S^{n-2}} \int_{S^{n-2}}|\nu\cdot\theta|\,d\mu(\theta) \ge \lambda,
\end{equation}
for some~$\lambda>0$.
To prove it, we observe that
if~$\nu=(\nu_1,\dots,\nu_{n-1})\in S^{n-2}$,
we have that~$|\nu_j|\ge (n-1)^{-1/2}$, for at least one~$j\in\{1,\dots,n-1\}$.
Up to relabeling variables, we assume that~$j=1$,
and thus
\begin{eqnarray*}
\int_{S^{n-2}}|\nu\cdot\theta|^{2s}\,d\mu(\theta) &\ge&\frac{a_1\,c_{N_1,s_1}}{2}
\int_{S^{N_1-1}} |\nu_1 \theta_1|^{2s}\,d\mu(\theta)\\
&\ge& \frac{a_1\,c_{N_1,s_1}}{2\,(n-1)^s}
\int_{S^{N_1-1}} |\theta_1|^{2s}\,d\mu(\theta),
\end{eqnarray*}
which proves~\eqref{HasUIYa66}.

In addition,
$$ \mu(S^{n-2})\le
\sum_{i=1}^{m-1} \frac{a_i\,c_{N_i,s_i}}{2} {\mathcal{H}}^{N_i-1}(S^{N_i-1}) <+\infty.$$
{F}rom this and~\eqref{HasUIYa66}, we conclude that condition~(1.2)
in~\cite{SRO} is satisfied. Accordingly, we can exploit
Theorem~1.1(a) in~\cite{SRO} and conclude that
\begin{eqnarray*}
\|\tilde u\|_{C^\gamma(B_{3/4}^{n-1})} &\le& C\,\big(
\|\tilde u\|_{L^\infty(\R^n)} +\|\tilde f\|_{L^\infty(B_{4/5}^{n-1})} \big) \\&\le&
C\,\big(
\|u\|_{L^\infty(\R^n)} +\|f\|_{L^\infty(B_1)}+ 
\|\partial^2_{x_n} u\|_{L^\infty(B_{97/100})}
\big).\end{eqnarray*}
This and~\eqref{97iuhtYUIO99-BIS} imply that
$$ \|\tilde u\|_{C^\gamma(B_{3/4}^{n-1})}\le
C\,\big(
\|\partial_{x_n} f\|_{L^\infty(B_1)}+
\|f\|_{L^\infty(\R^n)}+\|u\|_{L^\infty(\R^n)}
\big).$$
This gives the desired regularity in the set of variables~$(x_1,\dots,x_{n-1})$.
The regularity in the last variable follows from~\eqref{97iuhtYUIO99-BIS}
and so the proof of Corollary~\ref{RRO} is complete.

\section{Proof of Theorem \ref{MAIN 2}}\label{S:3}

\subsection{A cutoff argument}

The purpose of this section is to localize the estimate of
Theorem~\ref{MAIN} by using a cutoff function.
As customary in the fractional problems, regularity estimates
cannot be completely localized, due to nonlocal effect,
nevertheless our objective is to give quantitative bounds
on the contribution ``coming from infinity''. 
For this scope, we use the notation~$s_{\min}:=\min\{s_1,\dots,s_n\}$
and~$s_{\max}:=\max\{s_1,\dots,s_n\}$
(a similar notation will also be exploited in the sequel for
$a_{\min}:=\min\{a_1,\dots,a_m\}$ and $a_{\max}:=\max\{a_1,\dots,a_m\}$).

\begin{lemma}\label{CURT}
Let~$R\ge1$.
If~$w$ vanishes identically in~$(-3R,3R)^n$, then
$$ \|Lw\|_{L^\infty((-R,R)^n)}\le C_o
\int_{2R}^{+\infty}
\frac{ \|w\|_{L^\infty(B_{2\rho}\setminus
B_{{\rho}/{2}})} }{\rho^{1+2s_{\min}}}\,d\rho ,$$
where
\begin{equation}\label{C0}
C_o:= 2\sum_{i=1}^n a_i c_{N_i,s_i}\,{\mathcal{H}}^{N_i-1}(S^{N_i-1}).\end{equation}
\end{lemma}

\begin{proof} Let~$x\in (-R,R)$. 
We claim that
\begin{equation}\label{X1R}
|(-\Delta_{X_{i}})^{s_i} w(x)|\le 
2c_{N_i,s_i}\,\,{\mathcal{H}}^{N_i-1}(S^{N_i-1})\,\int_{2R}^{+\infty}
\frac{ \|w\|_{L^\infty(B_{2\rho}\setminus
B_{{\rho}/{2}})} }{\rho^{1+2s_i}}\,d\rho
\end{equation}
for each~$i\in\{1,\dots,m\}$.
To prove this, we notice that if~$s_i=1$ then
the fact that~$w$ vanishes identically
in a neighborhood of~$x$ implies
that~$-\Delta_{X_{i}} w(x)=0$, and so~\eqref{X1R} is obvious
in this case. Thus, we can
suppose that~$s_i\in(0,1)$, and we observe that, if~$y^{(i)}\in[-2R,2R]^{N_i}$
then~$x+y^{(i)}\in (-3R,3R)^n$ and so~$w(x+y^{(i)})=0$. {F}rom this,
it follows that
\begin{equation}\label{9ixqe} (-\Delta_{X_{i}})^{s_i} w(x) =
c_{N_i,s_i} \int_{\R^{N_i}\cap\{|y^{(i)}|>2R\}} \frac{-w(x+y^{(i)})-w(x-y^{(i)})
}{|y^{(i)}|^{N_i+2s_i}}\,dy^{(i)}.\end{equation}
Also, if~$|y^{(i)}|>2R$ then~$|y^{(i)}|>2|x|$, thus
\begin{eqnarray*}
&& |x\pm y^{(i)}|\ge |y^{(i)}|-|x| >\frac{|y^{(i)}|}{2}\\
{\mbox{and }}&&|x\pm y^{(i)}|\le|x|+|y^{(i)}|<2|y^{(i)}|,
\end{eqnarray*}
and so~$|w(x\pm y^{(i)})|\le \|w\|_{L^\infty(B_{2|y^{(i)}|}\setminus
B_{{|y^{(i)}|}/{2}})}$.
As a consequence of this and of~\eqref{9ixqe}, we obtain
\begin{eqnarray*}|(-\Delta_{X_{i}})^{s_i} w(x)|
&\le &
2c_{N_i,s_i}\int_{\R^{N_i}\cap\{|y|>2R\}}\frac{
\|w\|_{L^\infty(B_{2|y|}\setminus
B_{{|y|}/{2}})} }{|y|^{N_i+2s_i}}\,dy\\&=&
2c_{s_i}\,\,{\mathcal{H}}^{N_i-1}(S^{N_i-1})\,\int_{2R}^{+\infty}
\frac{ \|w\|_{L^\infty(B_{2\rho}\setminus
B_{{\rho}/{2}})} }{\rho^{1+2s_i}}\,d\rho,\end{eqnarray*}
which proves~\eqref{X1R}.
The desired claim then follows recalling~\eqref{L}
and adding up the estimate in~\eqref{X1R}.
\end{proof}

\begin{corollary}\label{co cut}
Let~$R\ge1$. There exists~$\eta_R\in C^\infty(\R^n)$ such that
\begin{eqnarray}
\label{WQ1} && {\mbox{$\eta_R=1$ in~$(-3R,3R)^n$,
$\eta_R=0$ in~$\R^n\setminus (-6R,6R)$, and}}\\
&&\qquad\label{WQ2}
\|Lu-L(\eta_R u) \|_{L^\infty((-R,R)^n)}\le
C_o
\int_{2R}^{+\infty}
\frac{ \|u\|_{L^\infty(B_{2\rho}\setminus
B_{{\rho}/{2}})} }{\rho^{1+2s_{\min}}}\,d\rho, 
\end{eqnarray}
with~$C_o$ as in~\eqref{C0}.
\end{corollary}

\begin{proof} Let~$\eta_o\in C^\infty(\R,\,[0,1])$ with~$\eta_o=1$
in~$(-1,1)$ and~$\eta_o=0$ outside~$(-2,2)$.
Let
$$\eta_R(x):=\prod_{i=1}^n\eta_o\left(\frac{x_i}{3R} \right).$$
Then $\eta_R$ satisfies \eqref{WQ1}. Also,
if we set~$w:=(1-\eta_R)u$, we have from~\eqref{WQ1}
that~$w=0$ in~$(-3R,3R)^n$. Thus,
the estimate in~\eqref{WQ2} follows by writing~$u-\eta_R u =w$,
using the linearity of the operator~$L$ and Lemma~\ref{CURT}.
\end{proof}

By combining Theorem~\ref{MAIN} and Corollary~\ref{co cut},
we can obtain a refined estimate in which the ``contribution
from infinity'' in the right hand side of~\eqref{STIMA:MAIN}
is weighted ``ring by ring'':

\begin{theorem}\label{MAIN WEIGHTED}
Let $R\ge1$ and~$f:B_{6\sqrt{n}R}\to\R$.
Let~$u:\R^n\to\R$ be a solution of~$Lu=f$ in~$B_{6\sqrt{n}R}$.
Then, for any~$t\in (-\frac{R}{6\sqrt{n}},\frac{R}{6\sqrt{n}})$,
\begin{equation}\label{MAIN WEIGHTED-EQ}
\begin{split}
& \frac{|u(te_n)-u(-te_n)|}{|t|} \le
C\left(
R\,
\sup_{(x,t)\in B_R\times(0,R)}|f(x+te_n)-f(x-te_n)|
\right.\\&\left.\qquad\quad+
\frac{R \,\|u\|_{L^\infty(B_{6R})}}{R^{2s_{\min}}}
+ R\,\int_{2R}^{+\infty}
\frac{ \|u\|_{L^\infty(B_{2\rho}\setminus
B_{{\rho}/{2}})} }{\rho^{1+2s_{\min}}}\,d\rho \right)
,\end{split}\end{equation}
where $C>0$ here only depends on~$n$, $s_{\min}$, $s_{\max}$,
$a_{\min}$ and~$a_{\max}$.
\end{theorem}

\begin{proof} In this argument, we will take the freedom of renaming
constants as we please, line after line, by keeping the same name~$C$.
Using the notation of 
Corollary~\ref{co cut}, we define~$\tilde u:=\eta_R u$
and, for any~$x\in B_{6\sqrt{n}R}$,
$\tilde f(x):= L\tilde u(x)$. Let also~$\tilde g:=L\tilde u-Lu$.
By Corollary~\ref{co cut},
$$ \|\tilde g\|_{L^\infty((-R,R)^n)}\le
C\int_{2R}^{+\infty}
\frac{ \|u\|_{L^\infty(B_{2\rho}\setminus
B_{{\rho}/{2}})} }{\rho^{1+2s_{\min}}}\,d\rho.$$
By construction~$\tilde f=f+\tilde g$, therefore
\begin{eqnarray*}
&& \sup_{(x,t)\in (-\frac{R}{6\sqrt{n}},\frac{R}{6\sqrt{n}})^n
\times(0,\frac{R}{6\sqrt{n}})} 
|\tilde f(x+te_n)-\tilde f(x-te_n)|\\&&\qquad\le
\sup_{(x,t)\in (-\frac{R}{6\sqrt{n}},\frac{R}{6\sqrt{n}})^n
\times(0,\frac{R}{6\sqrt{n}})} 
|f(x+te_n)-f(x-te_n)| +
2 \|\tilde g\|_{L^\infty((-\frac{R}{3\sqrt{n}},\frac{R}{3\sqrt{n}})^n)}
\\ &&\qquad\le
\sup_{(x,t)\in B_R\times(0,R)}|f(x+te_n)-f(x-te_n)|
+ C\int_{2R}^{+\infty}
\frac{ \|u\|_{L^\infty(B_{2\rho}\setminus
B_{{\rho}/{2}})} }{\rho^{1+2s_{\min}}}\,d\rho.\end{eqnarray*}
Notice also that~$\tilde u=u$ in~$(-R,R)^n$
and~$\tilde u=0$ outside~$B_{6R}$.
Thus, by applying Theorem~\ref{MAIN}
(here with~$d_1=\dots=d_m=\frac{R}{3\sqrt{n}}$)
to the function~$\tilde u$, for any~$t\in(-\frac{R}{6\sqrt{n}},\frac{R}{6\sqrt{n}})$
we obtain
that
\begin{eqnarray*}
&& \frac{|u(te_n)-u(-te_n)|}{|t|} =
\frac{|\tilde u(te_n)-\tilde u(-te_n)|}{|t|} \\
&&\qquad\le
CR\,
\sup_{(x,t)\in (-\frac{R}{6\sqrt{n}},\frac{R}{6\sqrt{n}})^n
\times(0,\frac{R}{6\sqrt{n}})}
|\tilde f(x+te_n)-\tilde f(x-te_n)|
+ \frac{C R \,\|\tilde u\|_{L^\infty(\R^{n})}}{R^{2s_{\min}}}\\
&&\qquad\le
CR\,
\sup_{(x,t)\in B_R\times(0,R)}|f(x+te_n)-f(x-te_n)|
+ CR\,\int_{2R}^{+\infty}
\frac{ \|u\|_{L^\infty(B_{2\rho}\setminus
B_{{\rho}/{2}})} }{\rho^{1+2s_{\min}}}\,d\rho +
\frac{C R \,\|u\|_{L^\infty(B_{6R})}}{R^{2s_{\min}}}
,\end{eqnarray*}
as desired.\end{proof}

\subsection{Completion of the proof of Theorem~\ref{MAIN 2}}

Using L'H\^opital's Rule, we see that
\begin{eqnarray*}&& \lim_{R\to+\infty}
R\,\int_{2R}^{+\infty}
\frac{ \|u\|_{L^\infty(B_{2\rho}\setminus
B_{{\rho}/{2}})} }{\rho^{1+2s_{\min}}}\,d\rho=
\lim_{R\to+\infty} \frac{\displaystyle\int_{2R}^{+\infty}
\frac{ \|u\|_{L^\infty(B_{2\rho}\setminus
B_{{\rho}/{2}})} }{\rho^{1+2s_{\min}}}\,d\rho}{ R }
=\lim_{R\to+\infty}
\frac{ 
\displaystyle\frac{ \|u\|_{L^\infty(B_{4R}\setminus
B_{R})} }{(2R)^{1+2s_{\min}}}
}{ R^{-2} }
=0,\\
&&\qquad\qquad{\mbox{ and }}\;\lim_{R\to+\infty}
\frac{R \,\|u\|_{L^\infty(B_{6R})}}{R^{2s_{\min}}}
=0,\end{eqnarray*}
thanks to \eqref{XRT7}.
So, we can use Theorem~\ref{MAIN WEIGHTED}
and pass formula~\eqref{MAIN WEIGHTED-EQ}
to the limit as~$R\to+\infty$, and obtain that
$$\frac{|u(te_n)-u(-te_n)|}{|t|} =0,$$
for any fixed~$t\in\R$. This says that~$u(te_n)=u(-te_n)$
for any~$t\in\R$

Since the problem is translation invariant, we can apply the
argument above in the neighborhood of any point, so we obtain that
\begin{equation}\label{9as12s}
u(p+te_n)=u(p-te_n)\end{equation}
for any~$p\in\R^n$ and any~$t\in\R$

Now take any point~$x\in\R^n$ and any~$\rho\in\R$.
We take~$p:=x+\frac{\rho\,e_n}{2}$ and~$t:=\frac{\rho}{2}$.
Notice that~$ p-te_n = x $ and~$p+te_n= x+\rho e_n$,
therefore~\eqref{9as12s}
implies that~$u(x)=u(x+\rho e_n)$, which completes the proof
of Theorem~\ref{MAIN 2}.

\vfill \end{document}